\documentclass[11pt]{amsart}%
\usepackage{amsmath}
\usepackage{amssymb}
\usepackage{amsthm}
\usepackage[mathscr]{eucal}
\usepackage{mathrsfs}
\usepackage{psfrag}
\usepackage{fullpage}
\usepackage{color}
\usepackage{eucal}
\usepackage[dvips]{graphicx}
\usepackage{amsfonts}%
\usepackage{amsaddr}
\providecommand{\U}[1]{\protect\rule{.1in}{.1in}}
\newtheorem{proposition}{Proposition}[section]
\newtheorem{theorem}[proposition]{Theorem}

\newtheorem{definition}[proposition]{Definition}
\newtheorem{remark}[proposition]{Remark}

\newtheorem{condition}[proposition]{Condition}

\numberwithin{equation}{section}
\numberwithin{proposition}{section}

\begin{document}
\title{Non-asymptotic  performance analysis of importance sampling schemes for small noise diffusions}
\author{Konstantinos Spiliopoulos}
\address{Department of Mathematics and Statistics, Boston University, Boston, MA 02215}
\email{kspiliop@math.bu.edu}
\date{\today}
\maketitle

\begin{abstract}
In this note we develop a prelimit analysis of performance measures for importance sampling schemes related to small noise diffusion processes. In importance sampling the performance of any change of measure
 is characterized by its second moment. For a given change of measure, we characterize the second moment of the corresponding estimator as the solution to a PDE, which we analyze via a full asymptotic expansion with respect to the size of the noise and obtain a precise statement on its accuracy.
The main correction term to the decay rate of the second moment solves a transport equation that can be solved explicitly. The asymptotic expansion that we obtain identifies the source of possible poor performance of nevertheless asymptotically optimal importance sampling schemes and allows for more accurate comparison among competing importance sampling schemes.
\end{abstract}


\textbf{AMS}: 60F05, 60F10, 60G60

\section{Introduction}
Consider a $d$-dimensional  process $X^{\epsilon}\doteq \{X^{\epsilon}(s), t\leq s\leq T\}$ satisfying the stochastic differential equation (SDE)
\begin{equation}
dX^{\epsilon}(s)=b\left(X^{\epsilon}(s)\right)ds+\sqrt{\epsilon}\sigma\left(X^{\epsilon}(s)\right)dW(s),
\hspace{0.2cm} X(t)=x\label{Eq:Diffusion1}
\end{equation}
where $0<\epsilon\ll 1$ and $W(s)$ is a standard $d$-dimensional Wiener process.

Assume that we are interested in estimating a quantity like
\begin{equation}
\mathbb{E}_{t,x}\left[e^{-\frac{1}{\epsilon}h\left(X^{\epsilon}(T)\right)}\right]\qquad \textrm{ or }\qquad\mathbb{P}_{t,x}\left[X^{\epsilon}(T)\notin D\right]
\end{equation}
where $h(x):\mathbb{R}^{d}\mapsto\mathbb{R}$ is a given nonnegative, bounded and continuous function and $D$ an appropriate set.

It is well known that standard Monte Carlo methods perform poorly in the small noise limit in that the relative errors under a fixed computational cost grow rapidly as the event becomes more rare. A popular variance reduction technique in Monte Carlo simulation that addresses this problem is importance sampling. The simulation is done under an alternative probability measure, absolutely continuous with respect to the original one, such that the variance of the resulting unbiased estimator (under the new probability measure) is as small as possible in the limit as $\epsilon\downarrow 0$. Due to unbiasedness, minimizing the variance is equivalent to minimizing the second moment of the estimator.

The asymptotically  optimal change of measure is related to the gradient of the solution to a Hamilton-Jacobi-Bellman (HJB) equation. Under conditions, the gradient of the system exists in the classical sense and is continuous, and then one can guarantee asymptotic optimality. The authors in \cite{WeareVandenEijden2012} consider importance sampling schemes that use as change of measure the gradient of the associated HJB equation and prove their asymptotic optimality in regions of strong regularity (see Definition \ref{Def:RSR}), under the assumption that the gradient of the solution to the HJB equation is continuously differentiable.  However, such smoothness of the solution to the HJB equation is
usually not the case, as for example in the case of computing exit or hitting probabilities, see \cite{DupuisSpiliopoulosZhou}. Also, even if the problem is such that the gradient of the solution is smooth, it may be difficult or actually impossible to reliably compute it in practice, even numerically. As the numerical results of \cite{DupuisSpiliopoulosZhou} demonstrate, this is certainly the case for points that are not in regions of strong regularity.

For these reasons, a theory
has been developed in \cite{DupuisWang,DupuisWang2,DupuisSpiliopoulosWang}  that gives precise bounds on asymptotic performance for schemes that are based on subsolutions to this HJB equation
rather than to the solution itself; see also \cite{BlanchetGlynnLeder2012} for the closely related concept of the associated Lyapunov inequalities, that can be also used to find bounds of the performance measures in the prelimit. Such results have been used to test the efficiency of state-dependent importance sampling schemes  in the heavy tail and discrete setting, e.g., \cite{Blanchet2009,BlanchetGlynn2008,BlanchetGlynnLiu2007, DupuisWang2}. Counterexamples in the discrete, non-dynamic, setting, where seemingly reasonable importance sampling schemes perform poorly in practice can be found in \cite{glakou,glawan}. In the continuous, dynamic setting, similar issues for metastable problems  were investigated in \cite{DupuisSpiliopoulosZhou}.

The goal of this paper is to characterize the performance of a given importance sampling scheme (not necessarily associated with a subsolution of the related HJB equation) for small but fixed $\epsilon>0$. For a given change of measure,
 we analyze the second moment of the estimator under the alternative probability measure through an asymptotic expansion in terms of $\epsilon$. Under regularity
conditions, we obtain the main correction terms in the logarithmic asymptotics of the second moment for a given change of measure. This expansion characterizes the performance of
an importance sampling algorithm for a small but fixed $\epsilon>0$.

The key contribution of this note is the characterization of the second moment of an importance sampling estimator as the solution to a PDE and the accuracy of the approximation of its solution by the corresponding asymptotic expansion, see Theorems \ref{T:CharacterizationPDE} and \ref{T:RSR}. The authors in \cite{FlemingJames1992,FlemingSouganidis1986} discuss convergence of solutions to the log transform of this type of PDEs  as $\epsilon\downarrow 0$,
 they characterize the limit of such equations as the viscosity solution to a certain PDE and discuss its smoothness properties. We refer the interested reader to those articles and to the references therein for an extensive discussion on these issues. In this note, we make the discussion specific for the particular problem of interest, which is the second moment of importance sampling schemes. We obtain representation formulas for the lower order correction terms in the asymptotic expansion of its solution. It turns out that the main correction term to the decay rate of the second moment solves a transport equation that can be solved explicitly.

The analysis of the leading order terms in the asymptotic expansion shows that for certain problems it may be possible to have relatively poor performance for moderately small values of $\epsilon$ even if
a change of measure with good asymptotic performance is being used. As we shall see, the main correction term to the logarithmic asymptotics is related to the trace of the second derivative matrix
of the limiting logarithmic asymptotics. If the trace of that second derivative matrix term takes large values or if the time horizon under consideration is large, it may be the case that the correction term
has significant contribution to the second moment for a fixed choice of $\epsilon$, yielding poor performance of the estimator. This could be useful in quantifying the question of when a specific choice of a subsolution is better than another one, for small but a fixed range of $\epsilon>0$.

A partial motivation for this work is the investigation of importance sampling schemes for problems with rest points in \cite{DupuisSpiliopoulosZhou}. The authors study there the problem of exit from an asymptotically stable equilibrium point of a dynamical system by given time $T$. It is found there that, nearly asymptotically optimal changes of measures may actually perform poorly, which is due to the presence of rest points (i.e., in that case, the stable equilibrium point).  In the  setting of \cite{DupuisSpiliopoulosZhou}, the authors also construct changes of measure that behave optimally both asymptotically and non-asymptotically.

In the current paper, the goal is to provide a non-asymptotic analysis of importance sampling schemes in a general setting.
 This is useful
 (a): in  obtaining non-asymptotic performance characterization,  (b): in providing a way to mathematically compare the performance of different competing changes of
measure for a specific problem, and (c): in  understanding why asymptotically optimal importance sampling schemes may actually perform poorly in practice.

Lastly, we would like to emphasize that the asymptotic series expansion that we obtain is not tied to importance sampling schemes that are based on the solution or on a subsolution of the  associated  HJB equation. The asymptotic expansion is valid for any ''sufficiently smooth" change of measure. If the change of measure happens to be related to a smooth subsolution of the associated HJB equation, then one can recover the well known asymptotically optimal characterization of the related importance sampling scheme.

The paper is organized as follows. In Section \ref{S:AsymptoticPerformance} we establish notation, recall the relevant large deviations results and results on asymptotically optimal importance
sampling for small noise diffusions.  In Section \ref{S:PrelimitBehavior}, we analyze the prelimit behavior of the second moment of an unbiased estimator based on a given change of
measure by performing an asymptotic expansion of the solution to the PDE that it satisfies in terms of $\epsilon$. In Section \ref{S:Limit behavior}, we show how the well known asymptotic performance can be recovered from the asymptotic analysis. In Section \ref{S:Example} we present an example where the issues addressed in this paper appear. Concluding remarks are in Section \ref{S:Conclusions}.

\section{Notation and asymptotically optimal change of measure}\label{S:AsymptoticPerformance}

In this section we introduce notation and review the related large deviations results for small noise diffusions. Moreover, we review the problem of importance sampling  in the context of small noise diffusions as well as the role of subsolutions in the construction of asymptotically optimal importance sampling schemes, Subsection \ref{SS:IS}.

First, we establish some notation.
Typically $(t,x)$ denotes a
specific starting time and initial state. With an abuse of notation
$(t,x)$ will also be used at times to denote a generic point in $[0,T]\times
\mathbb{R}^{d}$ (the intended use will be clear from context).

We define
\[
\|v\|_{B} \doteq\sqrt{v^{T} B v}%
\]
for any $v\in\mathbb{R}^{d}$ and symmetric positive definite matrix $B$. When
$B$ is the identity matrix, $\|v\|_{B}$ is just the standard Euclidean norm
$\|v\|$. Our main assumption is
\begin{condition}\label{A:MainAssumption}
\begin{enumerate}
\item{The drift $b(x)$ is bounded and Lipschitz continuous.}
\item{The coefficients $\sigma$ are bounded, Lipschitz continuous, uniformly nondegenerate with the inverse being bounded as well.}
\end{enumerate}
\end{condition}

The action functional associated with the process (\ref{Eq:Diffusion1}) takes the form, see \cite{FW1},
\begin{equation}
S_{tT}(\phi)\doteq
\begin{cases}
\int_{t}^{T}\frac{1}{2}\left\Vert \dot{\phi}(s)-b(\phi(s))\right\Vert^{2}_{a^{-1}(\phi(s))} ds, &\phi\in\mathcal{AC}([t,T])\textrm{ and }\phi(t)=x\\
+\infty, &\textrm{otherwise}
\end{cases}\label{Eq:RateFunction}
\end{equation}
where $a(x)=\sigma(x)\sigma^{T}(x)$ and  $\mathcal{AC}\left([t,T]\right)$ represents the set of absolutely continuous functions on $[t,T]$ with values in $\mathbb{R}^{d}$.

Given bounded and continuous function $h(x):\mathbb{R}^{d}\mapsto\mathbb{R}$, consider the problem of estimating

\begin{equation}
\theta(\epsilon;t,x)\doteq\mathbb{E}_{t,x}\left[e^{-\frac{1}{\epsilon}h\left(X^{\epsilon}(T)\right)}\right]\label{Eq:EstimationTargetSmooth1}
\end{equation}
where $\mathbb{E}$ is the expectation operator associated with the reference probability measure $\mathbb{P}$. Defining
\begin{equation}
G^{\epsilon}(t,x)\doteq-\epsilon\ln\theta(\epsilon;t,x)\label{Eq:EstimationTarget2}
\end{equation}
the following large deviations result is well known, e.g. \cite{FW1,FlemingSoner}.

\begin{theorem}\label{T:LDPTheorem}
Assume Condition \ref{A:MainAssumption}. Then  for each $(t,x)\in [0,T]\times\mathbb{R}^{d}$ the following statement holds
\begin{equation}
\lim_{\epsilon\downarrow 0}G^{\epsilon}(t,x)=G(t,x)\doteq\inf_{\phi\in\mathcal{C}([t,T]), \phi(t)=x}\left[S_{tT}(\phi)+h(\phi(T))\right]\label{Eq:LDPTh}
\end{equation}
\end{theorem}

Lastly, for notational convenience we define the
operator $\cdot:\cdot$, where for two matrices
$A=[a_{ij}],B=[b_{ij}]$
\[
A:B\doteq\sum_{i,j}a_{ij}b_{ij}.
\]

\subsection{Asymptotically optimal importance sampling}\label{SS:IS}
Let us briefly review now the problem of importance sampling for estimating $\theta(\epsilon)$ for a given function $h$. Consider $\Gamma^{\epsilon}(t,x)$ a random
variable defined on some probability space with probability measure $\bar{\mathbb{P}}$ such that
\[
\bar{\mathbb{E}}\Gamma^{\epsilon}(t,x)=\theta(\epsilon;t,x),
\]
where $\bar{\mathbb{E}}$ is the expectation operator associated with
$\bar{\mathbb{P}}$. In other words $\Gamma^{\epsilon}(t,x)$ is an unbiased estimator of $\theta(\epsilon;t,x)$.

In Monte Carlo simulation, one generates a number of independent copies of
$\Gamma^{\epsilon}(t,x)$ and the estimate is the sample mean. The specific
number of samples required depends on the desired accuracy, which is measured
by the variance of the sample mean. Clearly, since the samples are independent
it suffices to consider the variance of a single sample. Then
unbiasedness implies that, minimizing the variance is equivalent to minimizing the second
moment. By Jensen's inequality
\[
\bar{\mathbb{E}}(\Gamma^{\epsilon}(t,x))^{2}\geq(\bar{\mathbb{E}}%
\Gamma^{\epsilon}(t,x))^{2}=\theta(\epsilon)^{2}.
\]
The large deviations principle of  Theorem \ref{T:LDPTheorem} implies that
\[
\limsup_{\epsilon\rightarrow0}-\epsilon\log\bar{\mathbb{E}}(\Gamma^{\epsilon
}(t,x))^{2}\leq2G(t,x),
\]
and thus $2G(t,x)$ is the best possible rate of decay of the second moment.
If
\[
\liminf_{\epsilon\rightarrow0}-\epsilon\log\bar{\mathbb{E}}(\Gamma^{\epsilon
}(t,x))^{2}\geq2G(t,x),
\]
then $\Gamma^{\epsilon}(t,x)$ achieves this best decay rate, and is said to be
\textit{asymptotically optimal}.

The choices of unbiased estimators $\Gamma^{\epsilon}(t,x)$ that we shall consider are as follows. Consider $\bar{u}(s)$, a sufficiently smooth function such that the change of measure
\begin{equation}
\frac{d\bar{\mathbb{P}}^{\epsilon}}{d\mathbb{P}}=\exp\left\{  -\frac
{1}{2\epsilon}\int_{t}^{T}\left\Vert \bar{u}(s)\right\Vert
^{2}ds+\frac{1}{\sqrt{\epsilon}}\int_{t}^{T}\left\langle \bar{u}%
(s),dW(s)\right\rangle \right\},\label{Eq:ChangeOfMeasure}
\end{equation}
defines the family of probability measures $\bar{\mathbb{P}}^{\epsilon}$. Then, by Girsanov's Theorem
\[
\bar{W}(s)=W(s)-\frac{1}{\sqrt{\epsilon}}\int_{t}^{s}\bar{u}(\rho)d\rho,~~~t\leq s\leq T
\]
is a Brownian motion on $[t,T]$ under the probability measure $\bar
{\mathbb{P}}^{\epsilon}$, and $X^{\epsilon}$ satisfies $X^{\epsilon}(t)=x$
and
\[
dX^{\epsilon}(s)=\left[b\left(  X^{\epsilon}(s)\right) +\sigma\left(  X^{\epsilon
}(s)\right)\bar{u}(s) \right]ds+\sqrt{\epsilon}\sigma\left(  X^{\epsilon
}(s)\right) d\bar{W}(s)
\]

Letting
\[
\Gamma^{\epsilon}(t,x)=\exp\left\{  -\frac{1}{\epsilon}h\left(X^{\epsilon}(T)\right)\right\}  \frac{d\mathbb{P}}{d\bar{\mathbb{P}}^{\epsilon}}(X^{\epsilon}),
\]
it follows easily that under $\bar{\mathbb{P}}^{\epsilon}$, $\Gamma^{\epsilon
}(t,x)$ is an unbiased estimator for $\theta(\epsilon)$. The performance of
this estimator is characterized by the decay rate of its second moment
\begin{equation}
Q^{\epsilon}(t,x;\bar{u})\doteq\bar{\mathbb{E}}^{\epsilon}\left[  \exp\left\{
-\frac{2}{\epsilon}h\left(X^{\epsilon}(T)\right)\right\}  \left(  \frac{d\mathbb{P}%
}{d\bar{\mathbb{P}}^{\epsilon}}(X^{\epsilon})\right)  ^{2}\right]
.\label{Eq:2ndMoment1}%
\end{equation}

As it is demonstrated in \cite{DupuisWang2,DupuisSpiliopoulosWang} the asymptotically optimal change of measure is associated with subsolutions to a certain class of Hamilton-Jacobi-Bellman (HJB) equations of the type (\ref{Eq:HJBequation1})
\begin{eqnarray}
U_{t}(t,x)+\left\langle b(x),\nabla_{x} U(t,x)\right\rangle-\frac{1}{2}\left\Vert\sigma^{T}(x)\nabla_{x} U(t,x)\right\Vert^{2}&=&0, (t,x)\in[0,T)\times \mathbb{R}^{d}
\nonumber\\
U(T,x)&=&h(x),\text{ for } x\in \mathbb{R}^{d}\label{Eq:HJBequation1}
\end{eqnarray}
It can also be shown that if $G$ defined in Theorem \ref{T:LDPTheorem} is regular enough, then $G$ is actually the unique, bounded, continuous viscosity solution of (\ref{Eq:HJBequation1}), see \cite{FlemingSoner}.

Let us now recall the notion of subsolutions.

\begin{definition}
\label{Def:ClassicalSubsolution} A function $\bar{U}(t,x):[0,T]\times
\mathbb{R}^{d}\mapsto\mathbb{R}$ is a classical subsolution to the HJB
equation (\ref{Eq:HJBequation1}) if

\begin{enumerate}
\item $\bar{U}$ is continuously differentiable,

\item $\bar{U}_{t}(t,x)+H(x,\nabla_{x}\bar{U}(t,x))\geq 0$ for every
$(t,x)\in [0,T]\times \mathbb{R}^{d}$, and

\item $\bar{U}(T,x)\leq h(x)$ for $x\in\mathbb{R}^{d}$.
\end{enumerate}
\end{definition}

The use of subsolutions for importance sampling often leads to imposing stronger regularity conditions that go beyond those of Definition \ref{Def:ClassicalSubsolution}. For presentation purposes
we shall assume Condition \ref{A:MainAssumption2}, which however is by no means the most economical and can be replaced by milder conditions and further effort. Nevertheless, we shall assume it because it guarantees that the feedback control used in importance sampling
is uniformly bounded, which means that a number of technicalities are avoided.
\begin{condition}\label{A:MainAssumption2}
$\bar{U}(t,x)\in\mathcal{C}^{1,2}([0,T]\times\mathbb{R}^{d})$ and the first and second derivatives in $x$ are uniformly bounded.
\end{condition}

The connection between subsolutions and performance of importance sampling schemes has been established in several papers, such as \cite{DupuisWang2,DupuisSpiliopoulosWang}.
In the present setting, we have the following Theorem regarding asymptotic optimality (Theorem 4.1 in \cite{DupuisSpiliopoulosWang}).

\begin{theorem}
\label{T:UniformlyLogEfficientRegime1}  Let $\{X^{\epsilon},\epsilon>0\}$ be
the unique strong solution to (\ref{Eq:Diffusion1}). Consider a bounded and continuous
function $h:\mathbb{R}^{d}\mapsto\mathbb{R}$ and assume Condition
\ref{A:MainAssumption}. Let $\bar
{U}(t,x)$ be a subsolution according to Definition \ref{Def:ClassicalSubsolution} such that Condition \ref{A:MainAssumption2} holds and define the control
$\bar{u}(t,x)=-\sigma(x)^{T}\nabla_{x} \bar{U}(t,x)$. Then
\begin{equation}
G(t,x)+\bar{U}(t,x)\leq \liminf_{\epsilon\rightarrow0}-\epsilon\ln Q^{\epsilon}(t,x;\bar{u})\leq 2G(t,x).
 \label{Eq:GoalRegime1Subsolution}%
\end{equation}
\end{theorem}

Since $\bar{U}$ is a subsolution, we get that $\bar{U}(s,y)\leq G(s,y)$ everywhere. This implies that  the scheme is asymptotically optimal if $\bar
{U}(t,x)=G(t,x)$ at the starting point $(t,x)$. Hence, any subsolution with value at the origin $(t,x)$ such that
$$0\ll\bar{U}(t,x)\leq G(t,x) $$ will have better asymptotic performance than that of standard Monte Carlo which corresponds to the zero subsolution.

\section{Non-asymptotic performance of a given change of measure}\label{S:PrelimitBehavior}

In Section \ref{S:AsymptoticPerformance}, we recalled the construction of asymptotically optimal importance sampling schemes for small noise diffusion processes. As we saw there, under regularity conditions, a change of measure based on the solution or on a subsolution to the associated HJB equation will be asymptotically optimal in that the bound of Theorem \ref{T:UniformlyLogEfficientRegime1} is achieved. However notice that this is a logarithmic bound  based on a large deviations analysis. This means that the pre-exponential terms are ignored. Nevertheless, the simulation is done for a fixed $\epsilon>0$ and how small $\epsilon$ should be so that the logarithmic asymptotics dominate the performance may depend on the particular problem. This is exactly the situation that we want to study. Performing an asymptotic expansion of the second moment $Q^{\epsilon}(t,x;\bar{u})$ in terms of $\epsilon$ we quantify this relation.

We proceed as follows. Consider $\bar{u}(t,x):[0,T]\times\mathbb{R}^{d}\mapsto \mathbb{R}^{d}$ to be a given sufficiently smooth function and consider the change of measure given by (\ref{Eq:ChangeOfMeasure}).  We start with the following key result that characterizes the second moment $Q^{\epsilon}(t,x;\bar{u})$ as a solution to a PDE.

For a function $f\in\mathcal{C}^{2}(\mathbb{R}^{d})$ we define the operator
\begin{equation*}
\mathcal{L}^{\epsilon}_{\bar{u}}f(x)=\left<\left(b(x)-\sigma(x)\bar{u}(t,x)\right),\nabla_{x} f(x)\right>+\frac{\epsilon}{2}\sigma(x)\sigma^{T}(x):\nabla^{2}_{x}f(x)
\end{equation*}

\begin{theorem}\label{T:CharacterizationPDE}
Let $\epsilon>0$ and assume Condition \ref{A:MainAssumption}. Consider $h:\mathbb{R}^{d}\mapsto\mathbb{R}$ to be continuous and bounded and $\bar{u}:[0,T]\times\mathbb{R}^{d}\mapsto\mathbb{R}^{d}$ to be uniformly bounded, continuous in $t$ and Lipschitz continuous in $x$. Assume that the following equation has a unique, $\mathcal{C}^{1,2}\left([0,T]\times \mathbb{R}^{d}\right)$ bounded solution
\begin{eqnarray}
& &\Phi^{\epsilon}_{t}(t,x)+\mathcal{L}^{\epsilon}_{\bar{u}}\Phi^{\epsilon}(t,x)+\frac{1}{\epsilon}\left\Vert\bar{u}(t,x)\right\Vert^{2}\Phi^{\epsilon}(t,x)=0, (t,x)\in[0,T)\times\mathbb{R}^{d}\nonumber\\
& &\hspace{2cm}\Phi^{\epsilon}(T,x)=e^{-\frac{2}{\epsilon}h(x)}, x\in\mathbb{R}^{d}\label{Eq:PDE2ndMnt}
\end{eqnarray}
and that the derivative of $\Phi^{\epsilon}(t,x)$ up to second order with respect to $x$ and of first order with respect to $t$ are bounded and continuous in $(t,x)\in[h,T]\times\mathbb{R}^{d}$ for every $h\in(0,T)$.
Then we have that,  $\Phi^{\epsilon}(t,x)=Q^{\epsilon}(t,x;\bar{u})$.
\end{theorem}
\begin{proof}
We start by proving an equivalent representation for the second moment defined in (\ref{Eq:2ndMoment1}). In particular, we prove that
\begin{eqnarray}
Q^{\epsilon}(t,x;\bar{u})  &=&\bar{\mathbb{E}}^{\epsilon}\left[
\exp\left\{  -\frac{2}{\epsilon}h(X^{\epsilon}(T))\right\}  \left(
\frac{d\mathbb{P}}{d\bar{\mathbb{P}}^{\epsilon}}(X^{\epsilon})\right)
^{2}\right] \nonumber\\
&  =&\mathbb{E}\exp\left\{  -\frac{2}{\epsilon}h(\bar{X}^{\epsilon}%
(T))+\frac{1}{\epsilon}\int_{t}^{T}\left\Vert \bar{u}\left(  s,\bar
{X}^{\epsilon}(s)\right)  \right\Vert
^{2}ds\right\}\label{Eq:SecondMoment_representation}.
\end{eqnarray}
 where  $\bar{X}^{\epsilon}$ satisfies $\bar{X}^{\epsilon}(t)=x$ and
\begin{equation*}
d\bar{X}^{\epsilon}(s)   =\left[b  \left(\bar{X}^{\epsilon}(s)\right)-\sigma  \left(\bar{X}^{\epsilon}(s)\right)\bar{u}\left(  s,\bar
{X}^{\epsilon}(s)\right)  \right]ds+   \sqrt{\epsilon}\sigma  \left(\bar{X}^{\epsilon}(s)\right)dW(s).
\end{equation*}

Due to our assumptions $\bar{u}(s,x)$ is uniformly bounded. Therefore,  Girsanov's
theorem implies that
\[
\frac{d\mathrm{Q}}{d\mathbb{P}}\doteq\exp\left\{  -\frac{1}{\sqrt{\epsilon}%
}\int_{t}^{T}\left\langle \bar{u}\left(  s,\bar
{X}^{\epsilon}(s)\right),dW(s)\right\rangle -\frac{1}{2\epsilon}\int
_{t}^{T}\left\Vert \bar{u}\left(  s,\bar
{X}^{\epsilon}(s)\right)\right\Vert ^{2}ds\right\}
\]
defines a new probability measure $\mathrm{Q}$ under which
\[
\hat{W}(s)=W(s)+\frac{1}{\sqrt{\epsilon}}\int_{t}^{s}\bar{u}\left(  \rho,\bar
{X}^{\epsilon}(\rho)\right)d\rho
\]
is a Brownian motion. For notational convenience we write $\bar{u}(s)=\bar{u}(s,X(s))$, even though the particular process $X$ may change from line to line depending on the measure under which it is considered. Therefore, $\bar{X}^{\epsilon}$ under $\mathbb{P}$ has the same
distribution as $X^{\epsilon}$ under $\mathrm{Q}$. This implies
\begin{align*}
\lefteqn{\mathbb{E}\exp\left\{  -\frac{2}{\epsilon}h(\bar{X}^{\epsilon
}(T))+\frac{1}{\epsilon}\int_{t}^{T}\left\Vert \bar{u}(s,\bar{X}^{\epsilon
}(s))\right\Vert ^{2}ds\right\}  }\\
&  =\mathbb{E}^{\mathrm{Q}}\exp\left\{  -\frac{2}{\epsilon}h(X^{\epsilon
}(T))+\frac{1}{\epsilon}\int_{t}^{T}\Vert \bar{u}(s)\Vert^{2}ds\right\} \\
&  =\mathbb{E}\exp\left\{  -\frac{2}{\epsilon}h(X^{\epsilon}(T))+\frac
{1}{2\epsilon}\int_{t}^{T}\Vert \bar{u}(s)\Vert^{2}ds-\frac{1}{\sqrt{\epsilon}}%
\int_{t}^{T}\left\langle \bar{u}(s),dW(s)\right\rangle \right\}\\
&= \lefteqn{\bar{\mathbb{E}}^{\epsilon}\exp\left\{  -\frac{2}{\epsilon
}h(X^{\epsilon}(T))+\frac{1}{\epsilon}\int_{t}^{T}\Vert \bar{u}(s)\Vert^{2}%
ds-\frac{2}{\sqrt{\epsilon}}\int_{t}^{T}\left\langle \bar{u}(s),dW(s)\right\rangle
\right\}  }\\
&  =\bar{\mathbb{E}}^{\epsilon}\left[  \exp\left\{  -\frac{2}{\epsilon
}h(X^{\epsilon}(T))\right\}  \left(  \frac{d\mathbb{P}}{d\bar{\mathbb{P}%
}^{\epsilon}}(X^{\epsilon})\right)  ^{2}\right].\rule{4cm}{0cm}%
\end{align*}
which completes the proof of the equivalent representation for the second moment.

With expression (\ref{Eq:SecondMoment_representation}) at hand, the result follows by applying It\^{o}'s formula to the process $Y_{s}=e^{\frac{1}{\epsilon}\int_{t}^{s}\left\Vert\bar{u}(r,\bar{X}(r))\right\Vert^{2}dr}\Phi^{\epsilon}(s,\bar{X}(s))$. In particular, by
It\^{o} formula, see for example \cite{KaratzasShreve}, we have that
\begin{align}
&e^{\frac{1}{\epsilon}\int_{t}^{T}\left\Vert\bar{u}(r,\bar{X}^{\epsilon}(r))\right\Vert^{2}dr}\Phi^{\epsilon}(T,\bar{X}^{\epsilon}(T))=\Phi^{\epsilon}(t,x)+\nonumber\\
&+\int_{t}^{T}\sqrt{\epsilon}e^{\frac{1}{\epsilon}\int_{t}^{r}\left\Vert\bar{u}(\rho,\bar{X}^{\epsilon}(\rho))\right\Vert^{2}d\rho}\left<\nabla_{x} \Phi^{\epsilon}(r,\bar{X}^{\epsilon}(r)),\sigma(r,\bar{X}^{\epsilon}(r)) dW(r)\right>\nonumber\\
& +\int_{t}^{T}e^{\frac{1}{\epsilon}\int_{t}^{r}\left\Vert\bar{u}(\rho,\bar{X}^{\epsilon}(\rho))\right\Vert^{2}d\rho}\left[\Phi^{\epsilon}_{t}(r,\bar{X}^{\epsilon}(r))+\mathcal{L}_{\bar{u}}\Phi^{\epsilon}(r,\bar{X}^{\epsilon}(r))+\frac{1}{\epsilon}\left\Vert\bar{u}(r,\bar{X}^{\epsilon}(r))\right\Vert^{2}\Phi^{\epsilon}(r,\bar{X}^{\epsilon}(r))\right]dr\nonumber
\end{align}
By our assumptions, the stochastic integral term is a square integrable martingale and has mean zero. Thus, taking expected value
and  recalling that $\Phi^{\epsilon}$ satisfies (\ref{Eq:PDE2ndMnt}), we conclude the proof of the theorem.
\end{proof}

It is convenient to consider the function
$\Psi^{\epsilon}(t,x)=-\epsilon \log \Phi^{\epsilon}(t,x)$. It is easy
to see that $\Psi^{\epsilon}$ satisfies the equation
\begin{align}
&\Psi^{\epsilon}_{t}(t,x)+\mathcal{L}^{\epsilon}_{\bar{u}}\Psi^{\epsilon}(t,x)-\frac{1}{2}\left\Vert \sigma^{T}(x)\nabla_{x}\Psi^{\epsilon}(t,x)\right\Vert^{2}-\left\Vert\bar{u}(t,x)\right\Vert^{2}=0, (t,x)\in[0,T)\times\mathbb{R}^{d}\nonumber\\
&\hspace{2cm}\Psi^{\epsilon}(T,x)=2h(x), x\in\mathbb{R}^{d}\label{Eq:PDE2ndMntPsi}
\end{align}

Since $\left\Vert \bar{u}\right\Vert$ and $h$ are assumed to be bounded, the maximum principle for parabolic PDE's implies that $|\Psi^{\epsilon}(t,x)|\leq C$ uniformly in $(t,x)\in[0,T]\times\mathbb{R}^{d}$ and $\epsilon\in(0,1)$. In fact one can show, as in \cite{FlemingSoner}, that uniformly in compact subsets of $[0,T]\times\mathbb{R}^{d}$, we have that $\Psi^{\epsilon}\rightarrow v_{0}$, where $v_{0}$ is the unique bounded, continuous, viscosity solution of the Hamilton-Jacobi-Bellman equation
\begin{align}
&\partial_{t}v_{0}+\left<\left(b(x)-\sigma(x)\bar{u}(t,x)\right),\nabla_{x} v_{0}\right>-\frac{1}{2}\left\Vert \sigma^{T}(x)\nabla_{x}
v_{0}(t,x)\right\Vert^{2}-\left\Vert\bar{u}(t,x)\right\Vert^{2}=0, (t,x)\in[0,T)\times\mathbb{R}^{d}\nonumber\\
&\hspace{2cm}v_{0}(T,x)=2h(x), x\in\mathbb{R}^{d}\label{Eq:LDPterm}
\end{align}

We want to derive an asymptotic expansion for $\Psi^{\epsilon}(t,x)$ in terms of $\epsilon>0$. This will give an asymptotic expansion for $\Phi^{\epsilon}(t,x)$. At this point, we recall the notion of a region of strong regularity with respect to a vector $\beta$; see for example \cite{FlemingJames1992}.

\begin{definition}\label{Def:RSR}
We say that $N$ is a region of strong regularity with respect to a vector $\beta=\beta(s,x)$ with $(s,x)\in[0,T]\times\mathbb{R}^{d}$ if the following hold.
\begin{enumerate}
\item{$N$ is an open subset of $[0,T)\times\mathbb{R}^{d}.$}
\item{$\beta(s,\cdot)$ is Lipschitz for each fixed $s\in[0,T].$}
\item{Consider the differential equation
\begin{equation*}
\dot{y}(s)=\beta(s,y(s)), s\in[t,T) \textrm{ and }y(t)=x
\end{equation*}
Clearly $y(s)=y(s;t,x)$. For $(t,x)\in N$ define $\sigma(t,x)=\inf\left\{s>t: (s,y(s))\notin N\right\}$. Let us also define $z(t,x)=(\sigma(t,x),y(\sigma(t,x);t,x))$ and $\gamma(t,x)=\left\{(s,y(s)): t\leq s\leq\sigma(t,x)\right\}$.  Define the set $\Gamma=\left\{z(t,x):(t,x)\in N\right\}$.

Then, $\partial N=\Gamma\bigcup\Delta$, where $\Gamma$ is relatively open subset of $\partial N$, a $\mathcal{C}^{\infty}$ manifold, such that $\gamma(t,x)$ crosses $\Gamma$   nontangentially.

}
\end{enumerate}
\end{definition}

Assume the formal asymptotic expansion
\begin{equation*}
\Psi^{\epsilon}(t,x)=v_{0}(t,x)+\epsilon
v_{1}(t,x)+\epsilon^{2}v_{2}(t,x)+\cdots+\epsilon^{n}v_{n}(t,x)+O(\epsilon^{n})
\end{equation*}

In regions of strong regularity one can prove that this expansion is valid. In particular, we have the following theorem.

\begin{theorem}\label{T:RSR}
Assume that $h\in \mathcal{C}_{b}\left(\mathbb{R}^{d}\right)\bigcap\mathcal{C}^{\infty}\left(\mathbb{R}^{d}\right)$ and $\left\Vert \bar{u}\right\Vert\in \mathcal{C}_{b}\left([0,T]\times\mathbb{R}^{d}\right)\bigcap\mathcal{C}^{\infty}\left([0,T]\times\mathbb{R}^{d}\right)$. Let $N$ be a region of strong regularity with respect to $$\beta(s,x)=b(x)-\sigma(x)\bar{u}(s,x)-a(x)\nabla_{x} v_{0}(s,x)$$ such that $v_{0}\in\mathcal{C}^{\infty}(\bar{N})$. Then for every $n=0,1,2,\cdots$ we have uniformly on compact subsets of $N\bigcup\left(\partial N\bigcap \left(\{T\}\times\mathbb{R}^{d}\right)\right)$ that

\begin{equation*}
\Psi^{\epsilon}(t,x)=v_{0}(t,x)+\epsilon
v_{1}(t,x)+\epsilon^{2}v_{2}(t,x)+\cdots+\epsilon^{n}v_{n}(t,x)+O(\epsilon^{n})
\end{equation*}
as $\epsilon\downarrow 0$. The main term in the asymptotic expansion, $v_{0}$ satisfies (\ref{Eq:LDPterm}). Moreover,
\begin{align}
&
\partial_{t}v_{1}+\left<\left(b(x)-\sigma(x)\bar{u}(t,x)\right),\nabla_{x} v_{1}\right>+\frac{1}{2} \alpha(x):\nabla^{2}_{x}v_{0}-\left<\sigma^{T}(x)\nabla_{x} v_{0},\sigma^{T}(x)\nabla_{x} v_{1}\right>=0, (t,x)\in N\nonumber\\
& \hspace{2cm}v_{1}^{\epsilon}(T,x)=0, (T,x)\in \left(\partial N\bigcap \left(\{T\}\times\mathbb{R}^{d}\right)\right)\nonumber
\end{align}

and for every $n\geq 2$
\begin{align}
&
\partial_{t}v_{n}+\left<\left(b(x)-\sigma(x)\bar{u}(t,x)\right),\nabla_{x} v_{n}\right>+\frac{1}{2} \alpha(x):\nabla^{2}_{x}v_{n-1}-\sum_{i+j=n,0\leq i<j}\left<\sigma^{T}(x)\nabla_{x} v_{i},\sigma^{T}(x)\nabla_{x} v_{j}\right>-\nonumber\\
& \hspace{3cm}-\frac{1}{2}\left\Vert \sigma^{T}(x) \nabla_{x}
v_{n/2}\right\Vert^{2}\chi_{\{n/2 \textrm{ is integer}\}}=0, (t,x)\in N\nonumber\\
& \hspace{2cm}v_{n}^{\epsilon}(T,x)=0, (T,x)\in \left(\partial N\bigcap \left(\{T\}\times\mathbb{R}^{d}\right)\right)\nonumber
\end{align}
\end{theorem}

\begin{proof}
The proof is similar to that of Theorem 5.1 in \cite{FlemingJames1992} and thus omitted. The equations that are satisfied by the correction terms in the asymptotic expansion can be formally found by plugging in the asymptotic expansion in (\ref{Eq:PDE2ndMntPsi}) and matching powers of $\epsilon$,
$\epsilon^{\kappa}$, for $\kappa=0,1,2,\cdots$.
\end{proof}

We conclude this section with a few useful remarks. Under sufficient regularity, see for example Theorem I.5.1 in \cite{FlemingSoner}, the leading order term $v_{0}(t,x)$ can be written as the value function
to a variational problem. To be precise
\begin{eqnarray}
v_{0}(t,x)&=&\inf_{\phi\in\mathcal{AC}([t,T]):\phi(t)=x}\left\{\int_{t}^{T}\left[\frac{1}{2}\left\Vert
\dot{\phi}(s)-\left(b(\phi(s))-\sigma(\phi(s))\bar{u}(s,\phi(s))\right)\right\Vert^{2}_{a^{-1}(\phi(s))}\right.\right.\nonumber\\
& &\left.\left.\hspace{3cm}-\left\Vert\bar{u}(s,\phi(s))\right\Vert^{2}
\right]ds
+2 h(\phi(T))\right\}\label{Eq:MainLogarithmicTerm}
\end{eqnarray}

Also, notice that the equation for $\kappa=1$ is a transport equation. Thus, we can explicitly solve for the correction term in the asymptotic expansion $v_{1}$:
\begin{equation}
v_{1}(t,x)=\int_{t}^{T}\frac{1}{2} \alpha(\psi_{1}(s;x)):\nabla^{2}v_{0}(s,\psi_{1}(s;x))ds\label{Eq:Psi1forFirstOrderApproximation0}
\end{equation}
where $\psi_{1}(s;x)$ for $s\geq t$ solves the ODE
\begin{equation}
\dot{\psi}_{1}(s;x)=b(\psi_{1}(s;x))-\sigma(\psi_{1}(s;x))\bar{u}(s,\psi_{1}(s;x))-a(\psi_{1}(s;x))\nabla v_{0}(s,\psi_{1}(s;x)),
\psi_{1}(t;x)=x.\label{Eq:Psi1forFirstOrderApproximation}
\end{equation}

Note that in a similar fashion we can solve for the next lower terms of the formal asymptotic expansion, $v_{i}, i\geq 2$ as well.

\section{Recovering the known asymptotically optimal performance  results}\label{S:Limit behavior}
Theorem \ref{T:RSR} states that for small values of $\epsilon$, the theoretical performance of an importance sampling scheme, based on a given control $\bar{u}(\cdot)$ (equivalently change of measure) and for a fixed starting point in a region of strong regularity, can be approximated, under smoothness assumptions, by
\begin{equation}
\tilde{\Phi}^{\epsilon}(t,x)=\exp\left\{-\frac{1}{\epsilon}v_{0}(t,x)-\int_{t}^{T}\frac{1}{2} a(\psi_{1}(s;x)):\nabla^{2}v_{0}(s,\psi_{1}(s;x))ds\right\}\label{Eq:PrelimitPerformance}
\end{equation}
where $v_{0}(t,x)$ satisfies (\ref{Eq:LDPterm})  and $\psi_{1}(s;x)$ satisfies the ODE (\ref{Eq:Psi1forFirstOrderApproximation}).
\begin{remark}\label{R:Signs}
Clearly the sign of $v_{1}(t,x)$ is of vital performance for the performance of a given importance sampling change of measure and one would like it to be positive. However, the point is that the situation where the rate of decay $v_{0}(t,x)$ is positive but the correction term $v_{1}(t,x)$ is negative can happen, as the example in Section \ref{S:Example} shows. In such a case the two terms compete and the question of which terms dominates which depends on the relative size of $\frac{1}{\epsilon} v_{0}(t,x)$ and of the correction term $v_{1}(t,x)$.
\end{remark}

In this section we show that the leading order term $v_{0}(t,x)$ matches the asymptotically optimal bounds of Theorem \ref{T:UniformlyLogEfficientRegime1}, if $\bar{u}(t,x)$ is chosen to be the solution or a subsolution to the associated HJB equation (\ref{Eq:HJBequation1}).

Let us assume that the actual solution to the related HJB equation is smooth enough and consider the control $\bar{u}(t,x)=-\sigma^{T}(x)\nabla_{x}G(t,x)$.  Recalling that $$G(t,x)=\inf_{\phi\in\mathcal{AC}[t,T],\phi(t)=x}\left\{\int_{t}^{T}\frac{1}{2}\left\Vert
\dot{\phi}(s)-b(\phi(s))\right\Vert^{2}_{a^{-1}(\phi(s))}ds+h(\phi(T))\right\},$$
and that $G(t,x)$ satisfies (\ref{Eq:HJBequation1}), we obtain
\begin{eqnarray}
v_{0}(t,x)&=&\inf_{\phi\in\mathcal{AC}[t,T],\phi(t)=x}\left\{\int_{t}^{T}\left[\frac{1}{2}\left\Vert
\dot{\phi}(s)-\left(b(\phi(s))-\sigma(\phi(s))\bar{u}(s,\phi(s))\right)\right\Vert^{2}_{a^{-1}(\phi(s))}\right.\right.\nonumber\\
& &\hspace{4cm}\left.\left.
-\left\Vert\bar{u}(s,\phi(s))\right\Vert^{2}
\right]ds
+2h(\phi(T))\right\}\nonumber\\
&=&2G(t,x).\nonumber
\end{eqnarray}

This is the known asymptotically optimal decay rate for the second moment (compare with Theorem \ref{T:UniformlyLogEfficientRegime1} by taking $\bar{U}=G$).
Assuming enough smoothness of $G(t,x)$, the leading order correction term in the approximation, $v_{1}(t,x)$, takes the form (\ref{Eq:Psi1forFirstOrderApproximation0}) with $v_{0}(t,x)=2G(t,x)$

Due to possible non-smoothness of the solution to the HJB equation or because it may be very difficult (which is usually the case) to compute it, the theory of subsolutions has been
developed, see \cite{DupuisWang2,DupuisSpiliopoulosWang}. Let us now study how $v_{0}(t,x)$ looks like when $\bar{u}(t,x)=-\sigma^{T}(x)\nabla_{x}\bar{U}(t,x)$, where $\bar{U}(t,x)$ is a subsolution to the associated HJB equation according to Definition \ref{Def:ClassicalSubsolution} and subject to Condition \ref{A:MainAssumption2}. In this case, we cannot obtain a closed form expression for $v_{0}(t,x)$ as it was done in the case of $\bar{U}(t,x)=G(t,x)$. However, using the subsolution property we can get the asymptotic bound obtained in Theorem \ref{T:UniformlyLogEfficientRegime1}. In particular, given an arbitrary $\phi\in\mathcal{AC}([t,T])$ with
$\phi(t)=x$, the subsolution property implies that
\begin{align*}
& -\langle\dot{\phi}(s)-b(\phi(s)),\nabla_{x}\bar{U}(s,\phi(s))\rangle
-\frac{1}{2}\langle\nabla_{x}\bar{U}(s,\phi(s)),a(\phi(s))\nabla_{x}\bar
{U}(s,\phi(s))\rangle\nonumber\\
&\hspace{6cm}\geq -\partial_{t}\bar{U}(s,\phi(s))-\langle\nabla_{x}\bar{U}(s,\phi(s)),\dot{\phi}(s)\rangle\nonumber\\
&\hspace{6cm}=-\frac{d}{ds}\bar{U}(s,\phi(s))  \nonumber
\end{align*}
Integrating both sides on $[t,T]$ and using  the terminal condition $\bar{U}(T,x)\leq h(x)$, we get
\begin{align*}
& -\int_{t}^{T}\left(\langle\dot{\phi}(s)-b(\phi(s)),\nabla_{x}\bar{U}(s,\phi(s))\rangle
+\frac{1}{2}\langle\nabla_{x}\bar{U}(s,\phi(s)),a(\phi(s))\nabla_{x}\bar
{U}(s,\phi(s))\rangle\right)ds \geq\nonumber\\
&\hspace{8cm}\geq -h(\phi(T))+\bar{U}(t,x)
\end{align*}
Expanding the square in the expression for $v_{0}(t,x)$ in (\ref{Eq:MainLogarithmicTerm}) and using the last display, we indeed see that
\begin{equation*}
v_{0}(t,x)\geq G(t,x)+\bar{U}(t,x)
\end{equation*}
which, as expected, is the asymptotic lower bound obtained in Theorem \ref{T:UniformlyLogEfficientRegime1}. 
\section{An illustrating example}\label{S:Example}

In this section, we present a simple example that demonstrates some of the issues mentioned in the introduction and can be explored via the analysis of Sections
\ref{S:PrelimitBehavior} and \ref{S:Limit behavior}. Consider the reversible process
\begin{equation}
dX^{\epsilon}(s)=-\nabla V\left(X^{\epsilon}(s)\right)ds+\sqrt{\epsilon}dW(s),
\hspace{0.2cm} X(0)=y\label{Eq:Diffusion2}
\end{equation}

Let us consider an open set $D$ and define $\tau_{\epsilon}$ to be the exit time of $X^{\epsilon}$ from $D$,
\[
\tau_{\epsilon}=\inf\{s>0:X^{\epsilon}(s)\notin D\}
\]
and we assume that $V\in\mathcal{C}^{2}(D)$ with bounded derivative for $x\in D$. We make the assumption now that the initial point $y\in D$ is the global minimum of $V(x)$ such that $\nabla V(y)=0$ and that $\nabla V(x)\neq 0$ for all $x\neq y$. Without loss of generality we further assume that $V(y)=0$. Thus the initial point $y$ is an asymptotically  stable equilibrium point for the stochastic dynamical system under consideration and the probability of going away from $y$ is a rare event probability. Let $0\leq \ell<L$, define $A_{\kappa}=\{x\in\mathbb{R}^{d}:V(x)=\kappa\}$ and
\[
D=\{x\in\mathbb{R}^{d}: \ell< V(x)<L\}.
\]

We want to estimate the probability
\[
  \mathbb{P}_{y}\left[X^{\epsilon} \text{ hits }A_{L} \text{ before hitting }A_{\ell} \text{ and before time } T \right].
\]

In this case we have that $h(x)=0$ if $x\in A_{L}$ and $h(x)=\infty$ if $x\in A_{\ell}$. The terminal cost $h(x)$ is neither bounded nor continuous, but as it is shown in \cite{DupuisSpiliopoulosWang} a smooth subsolution  will still satisfy Theorem \ref{T:UniformlyLogEfficientRegime1}. It is easy to see that $\bar{U}(x)=-2(V(x)-L)$ is a classical subsolution to the associated HJB equation, i.e., it satisfies Definition \ref{Def:ClassicalSubsolution}. Moreover, it is also easy to see that using the related control
\[
\bar{u}(t,x)=-\nabla \bar{U}(t,x)=2\nabla V(x)
\]
for the change of measure, results in essentially simulating under a new change of measure with reversed potential function in the domain of interest $D$. Let us analyze the performance of
such an importance sampling scheme, i.e., where under the new measure, the potential function has been reversed. This means that, under the new dynamics,
the rare event is no longer rare, because the dynamics do not push the trajectories towards the stable equilibrium point any more.

Based on Theorem \ref{T:UniformlyLogEfficientRegime1}, the lower bound characterizing the asymptotic performance  is $G(0,y)+\bar{U}(y)$, where
\[
G(0,y)=\inf_{\phi
\in\Lambda(0,y)}\left[  S_{0T}(\phi)+h(\phi(T))\right]  , \label{Eq:LDPTh}%
\]
where $S_{0T}(\phi)$ is given by (\ref{Eq:RateFunction}) and
\[
\Lambda(0,y)=\left\{  \phi\in\mathcal{C}([0,T]:\mathbb{R}^{d}):\phi
(0)=y,\phi(s)\in D\text{ for }s\in\lbrack 0,T],\phi(T)\in
\partial D\right\}.
\]

It is clear that $G(0,y)$ also depends on $T$. By the discussion of Section \ref{S:Limit behavior}, this means
\[
v_{0}(0,y)\geq G(0,y)+\bar{U}(y).
\]

It is shown in Chapter 4 of \cite{FW1}, as $T$ gets larger, $G(0,y)$ and $\bar{U}(y)$ get closer in value. Thus, by Theorem \ref{T:UniformlyLogEfficientRegime1} and for large enough $T$, the particular importance sampling scheme is asymptotically optimal.

However, it turns out that one can derive a non-asymptotic lower bound for the related second moment. The bound  indicates that the pre-exponential factors (correction terms) have significant contribution, sufficient to degrade the actual performance. The starting point to derive a lower bound for $\Phi^{\epsilon}(t,x)=Q^{\epsilon}(t,x;\bar{u})$ at the initial point $(t,x)=(0,y)$ with the control $\bar{u}(t,x)=2\nabla V(x)$ is the following stochastic control representation (we refer the reader to \cite{DupuisSpiliopoulosWang} for more details on the derivation and on the use of such representations in the asymptotic analysis of importance sampling schemes). Let $\mathcal{A}$ be the set of all $\mathfrak{F}_{t}-$progressively measurable $d-$dimensional processes $u\in\{u(s), s\in[0,T]\}$ that satisfy
\[
\mathbb{E}\int_{0}^{T}\left\Vert u(s)\right\Vert^{2}ds<\infty
\]
and for each fixed $\epsilon>0$ consider $\bar{X}^{\epsilon}$ to be the unique strong solution to the controlled SDE
\[
d\bar{X}^{\epsilon}(s)=\left[-\nabla V\left(\bar{X}^{\epsilon}(s)\right)+\left(u(s)-\bar{u}\left(\bar{X}^{\epsilon}(s)\right)\right)\right]ds+\sqrt{\epsilon}dW(s),
\hspace{0.2cm} \bar{X}(0)=y
\]
Then, with $\bar{\tau}_{\epsilon}$ denoting the first time of exit of $\bar{X}^{\epsilon}$ from the set $D$, we have the representation
\begin{equation}
\Psi^{\epsilon}(0,y)=-\epsilon\log Q^{\epsilon}(0,y;\bar{u})=\inf_{u\in\mathcal{A}}\mathbb{E}\left[\frac{1}{2}\int_{0}^{\bar{\tau}_{\epsilon}}\left\Vert u(s)\right\Vert^{2}ds-\int_{0}^{\bar{\tau}_{\epsilon}}\left\Vert \bar{u}\left(\bar{X}^{\epsilon}(s)\right)\right\Vert^{2}ds+\infty \chi_{\{T<\bar{\tau}_{\epsilon}\}}\right]\label{Eq:StoControlRep}
\end{equation}

Following similar argument as in \cite{DupuisSpiliopoulosZhou}, let us choose a particular admissible control $u(s)$ in the last display. Let $T$ be large and fix
 $0<K<T$ splitting the time interval into two parts $[0,T-K)$ and $[T-K,T]$. Set $u(s)=0$ for $s\in[0,T-K)$. It is easy to see that the resulting dynamics
of $\bar{X}$ process is stable for $s\in[0,T-K)$ and thus with high probability it will stay around the point $y$ for $s\in[0,T-K)$. In the second time interval,
 i.e., for $s\in[T-K,T]$ we choose a control that leads to escape prior to time $T$. The cost induced by such a control may depend on $K$, but not on $T$.
 Since, the representation in (\ref{Eq:StoControlRep}) takes the infimum over all $u\in\mathcal{A}$, the choice of a particular control results in an upper bound.
In particular, see also \cite{DupuisSpiliopoulosZhou}, we have  that there are positive constants $C_{1},C_{2}>0$, independent of $T$ or $\epsilon$, such that
\begin{equation}
Q^{\epsilon}(0,y;\bar{u})\geq e^{-\frac{1}{\epsilon}C_{1}+C_{2}(T-K)}.
\end{equation}

This lower bound on second moment indicates that if $T$ is large, one may need to go to considerably small values of  $\epsilon$, before the theoretically optimal
asymptotic performance is observed.
\section{Concluding remarks}\label{S:Conclusions}

The expression in  (\ref{Eq:PrelimitPerformance}) gives an approximation to the second moment of an unbiased estimator for a given change of measure for small but fixed $\epsilon>0$. It characterizes the leading order terms in the asymptotic expansion of the performance measure of an importance sampling estimator. Moreover, it shows that even in the case that the change of measure is chosen based on a subsolution, the correction term
$v_{1}(t,x)$ may dominate the main term of the asymptotics $-\frac{1}{\epsilon}v_{0}(t,x)$ degrading the performance for small but fixed values of $\epsilon>0$.
 A reason for that could be the consideration of big time horizon $T$ values.

Hence, for
certain choices of changes of measure, one may need to go to considerably smaller values of $\epsilon$ in order to observe good performance. An example where issues like these occur is studied in \cite{DupuisSpiliopoulosZhou}. As it is shown there the phenomenon of good theoretical asymptotic performance but poor actual performance is especially evident in problems with rest points.

The expression for the main correction term $v_{1}(t,x)$
 allows to compare importance sampling schemes  for small but fixed $\epsilon>0$. It also shows that the choice of change of measure does not only affect the leading order term in the asymptotics but also the correction terms. As it is demonstrated in the general framework considered here and shown in the specific example of Section \ref{S:Example}, the reason of poor actual performance of nevertheless theoretically asymptotically optimal changes of measures, lies in the behavior of the correction terms. The correction terms  can have significant contribution to the performance of the simulation algorithm, which is implemented for small but fixed non-zero strength of the noise. In this paper we have quantified in precise mathematical terms these issues.

\section{Acknowledgements}
The author was partially supported by the National Science Foundation (DMS 1312124).



\end{document}